\documentclass[a4paper,12pt]{article}
\usepackage[OT4]{fontenc}
\usepackage[a4paper,left=3cm,right=3cm,top=3.5cm,bottom=4.0cm]{geometry}

\usepackage{amsmath}
\usepackage{amsfonts}
\usepackage{color}
\usepackage{amssymb}
\usepackage{authblk}
\usepackage{enumitem}
\usepackage{caption}
\usepackage{subcaption}
\usepackage{epic}
\usepackage{eepic}
\usepackage{graphicx}
\usepackage{graphics}

\newtheorem{theorem}{Theorem}[section]
\newtheorem{lemma}{Lemma}[section]

\newtheorem{df}{Definition}[section]
\newtheorem{obs}{Observation}[section]

\newcommand{\cG}{\mathcal{G}}
\newcommand{\natplus}{\mathbb{N}^+}

\newenvironment{proof}{{\bf Proof.}}{\hspace*{\fill} \rule{2mm}{2mm} \par \hspace{0.1mm}}

\title{Twin Domination Number of Tournaments}
\author[1]{Dorota Osula}
\author[2]{Rita Zuazua}

\affil[1]{Faculty of Electronics, Telecommunications and Informatics \protect\\ Gda\'nsk University of Technology, 80-233 Gda\'nsk, Poland \protect\\
  \texttt{dorurban@student.pg.edu.pl}}
\affil[2]{Universidad Nacional Aut\'onoma de M\'exico, Mexico\protect\\
  \texttt{ritazuazua@ciencias.unam.com}}
\date{}

\begin{document}
\maketitle
\begin{abstract} 
Let $D=(V,A)$ be a digraph. A subset $S$ of $V$ is called a twin dominating set of $D$ if for every vertex $v\in V-S$, there exists vertices $u_1,u_2 \in S$ such that $(v,u_1)$ and $(u_2,v)$ are arcs in $D$. The minimum cardinality of a twin dominating set in $D$ is called the twin domination number of $D$ and is denoted by $\gamma ^{*}(D)$.

In \cite{ChDSS}, is defined the concept of upper orientable twin domination number of a graph $G$, 
$DOM^{*}(G)=\max\{ \gamma ^{*}(D)|D \ \text{is an orientation of G} \}.$ In \cite{AES}, it is conjectured that for the complete graph $K_n$ with $n\geq 8$, $DOM^{*}(K_n)=\left\lceil \frac{n+1}{2}\right\rceil$. In this work we prove $DOM^{*}(K_8)= DOM^{*}(K_9)= 4$ and establish new upper bounds for $DOM^{*}(K_n)$, disproving the same above conjecture for all $n \geq 8$. 

\end{abstract}



\section{Introduction}

Let $D=(V,A)$ be a digraph. For any vertex $v\in V$, the sets   $I_D(v)=\{ u |(u,v)\in A\}$ and $O_D(v)=\{ u|(v,u)\in A\}$  are called the inset and outset of $v$. The indegree and outdegree of $v$ are defined by $id_D(v)=|I_D(v)|$ and $od_D(v)=|O_D(v)|$. For any $S \subseteq V$, graph $\cG_D[S]$ is the subgraph of $D$ induced by the set of vertices $S$. The bottom index is omitted, when the digraph is clear from the context. We say that a set of vertices $S \subseteq V$ \emph{out-dominates} (\emph{in-dominates}) $V$ if for every $v \in V - S$ there exists $u \in S$ such that $(v,u) \in A$ ($(u,v) \in A$). A tournament is an orientation of a complete graph. 

\begin{df} Let $D=(V,A)$ be a digraph. A subset $S$ of $V$ is called a twin dominating set of $D$ if for every vertex $v\in V-S$, there exists vertices $u_1,u_2 \in S$ ($u_1$ and $u_2$ may be equal) such that $(v,u_1)$ and $(u_2,v)$ are arcs in $D$. The minimum cardinality of a twin dominating set in $D$ is called the twin domination number of $D$ and is denoted by $\gamma ^{*}(D)$. 
\end{df}

For different orientations $D_1$ and $D_2$ of a graph $G$, it is possible to have 
$\gamma ^{*}(D_1)\not= \gamma ^{*}(D_2)$. In {\sc Chartrand} et al.~\cite{ChDSS}, the authors defined the concept of upper orientable twin domination number of a graph $G$, 
\begin{align*}
DOM^{*}(G)=\max\{ \gamma ^{*}(D)|D \ \text{is an orientation of }G \}.
\end{align*}

{\sc Arumugam} et al.~\cite{AES} proved that for $1 \leq n \leq 7$ and the complete graph $K_n$, $DOM^{*}(K_n) = \left\lceil \frac{n+1}{2}\right\rceil$ and they conjectured that for $n\geq 8$, $DOM^{*}(K_n)=\left\lceil \frac{n+1}{2}\right\rceil$. We prove in Section~$2$ that for $n=8$, $DOM^{*}(K_8)= 4$ and in Section~$3$ that $DOM^{*}(K_9)= 4$. Then in Section~$4$ we expand these results by proving $DOM^{*}(K_n)\leq\left\lfloor \frac{n}{2}\right\rfloor$ for all $n\geq 8$. In Section~$5$ the upper bound of $2\left\lceil \log_2{(n-1)}\right\rceil$ is shown and the summary of all results presented.

\section{$DOM^*(K_8) = 4$}\label{sec:for8}

In this section we prove the exact value of the upper orientable twin domination number of the graph $K_8$ to be equal to $4$, i.e. $DOM^*(K_8) = 4$. The following observation is a consequence of the results given in {\sc Arumugam} et al.~\cite{AES}. 

\begin{obs} \label{sink-source}
Let $T$ be a tournament of order $n\geq 3 .$ From \textit{Theorem 2.6} of {\sc Arumugam} et al.~\cite{AES}, if $T$ contains at least one vertex 
$u \in V(T)$ such that $id (u) = 0$ or $od (u) = 0$, then   $\gamma ^{*}(T) \leq \left\lceil \log_2{(n-1)}+1\right\rceil $. In the case of $n\geq 8$,  $\gamma ^{*}(T)\leq \left\lfloor \frac{n}{2}\right\rfloor $. 
\end{obs}

Above observation assures us that if a tournament of order $8$ contains a sink or source vertex then $DOM^*(K_8) \leq 4$. In the next lemma we prove that this result holds also when a tournament contains a vertex of in- or outdegree $2$.

\begin{lemma} \label{thm:eight_two} Let $T$ be an orientation of $K_8$. If there exists $v\in V(T)$ such that $ id (v) = 2$ or $od (v) = 2$, then $\gamma ^{*}(T)\leq 4$. 
\end{lemma}

\begin{proof}
Suppose there exists a vertex $v \in V(T)$ such that $id(v)=2$, $I(v)=\{ i, i'\}$ and $O(v)=\{ o_1, o'_1, o_2, o'_2,z \}$. Without loss of generality we can assume that
 the arcs $(i,i'), (o_{1}',o_1), (o_{2}',o_2) \in A(T)$.
\begin{enumerate}
\item If the arcs $ (z,i), (z,o_1) $ or $(z, o_2)$ are in $A(T)$, then $S = \{v, i, o_1, o_2\}$ is a twin dominating set of $T$. So we can assume that the arcs 
$(i,z), (o_1,z), (o_2,z) \in A(T)$. See Figure~\ref{fig:example8_two}.
\item If the arc $(o_{1}',o_{2}') \in A(T)$, then the set $S = \{v, i, z, o_2'\}$ is a twin dominating set of $T$. If the arc $(o_{2}',o_{1}') \in A(T)$, then the set 
$S = \{v, i, z, o_1'\}$ is a twin dominating set of $T$. 
\end{enumerate}

Therefore, if $id(v)=2$, $\gamma ^{*}(T)\leq 4$. The case $od (v) = 2$, is symmetric. 

\end{proof}

\begin{figure}[htb]
\begin{center}
\centering
  \includegraphics[scale=0.6]{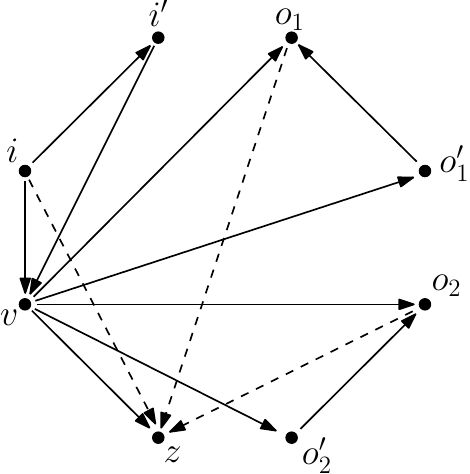}
	\end{center}
\caption{Illustration of Step 1 from Lemma~\ref{thm:eight_two} for the tournament $T$ of order $8$, which has a vertex $v$ of $id(v)=2$; \textit{dashed arrows} denote new added edges in Step~1.}
\label{fig:example8_two}
\end{figure}

In the next considered case a tournament of order $8$ contains a vertex of in- or outdegree $1$.

\begin{lemma}\label{thm:eight_one} Let $T$ be an orientation of $K_8$. If there exists $v\in V(T)$ such that $ id (v) = 1$ or $od (v) = 1$, then $\gamma ^{*}(T)\leq 4$.
\end{lemma}

\begin{proof} Let $v\in V(T)$ such that  $id(v)=1$  with $I(v)=\{ z \}$ and 
$O(v)=\{ o_1,o_1',  o_2, \allowbreak o_2', o_3,o_3' \}$. Without loss of generality we can assume that the arcs 
$(o_1',o_1),$ $(o_2',o_2),$ $(o_3',o_3) \in A(T)$. 
\begin{enumerate}
\item If $(o_1,z), (o_2,z)$ or $(o_3,z)$ are in $A(T)$, then $S=\{ v, o_1, o_2, o_3\}$ is a twin dominating set of $T$. 
So we can assume that $\{ v, o_1,o_2,o_3\} \subseteq O(z)$. 

\item By Observation~\ref{sink-source}, we can assume that $id(z)\neq 0$. Assume, without lost of generality, that 
$(o_3',z) \in A(T)$. If one of the arcs $(o_3,o_1)$ or $(o_3,o_2)$ are in $T$, then  $S = \{v, z, o_1, o_2\}$ is a twin dominating set of $T$. So let us assume that 
$(o_1,o_3), (o_2,o_3) \in A(T)$. See Figure~\ref{fig:example8_one}. 

\item If $(o_3,o_1')$ and  $(o_3,o_2')$ are arcs in $T$, then $od(o_3) = 2$ and from Lemma~\ref{thm:eight_two}, $\gamma^*(T) \leq 4$. On the other hand, 
if $(o'_1,o_3)$ (resp. $(o'_2,o_3)$) is an arc in $T$, then  $S=\{v,z,o_2,o_3\}$  (resp. $S=\{v,z,o_1,o_3\}$) is a twin dominating set of $T$.
\end{enumerate}

The case $od (v)=1$ is symmetric. 
\end{proof}

 \begin{figure}[htb]
\centering
\begin{center}
\includegraphics[scale=0.6]{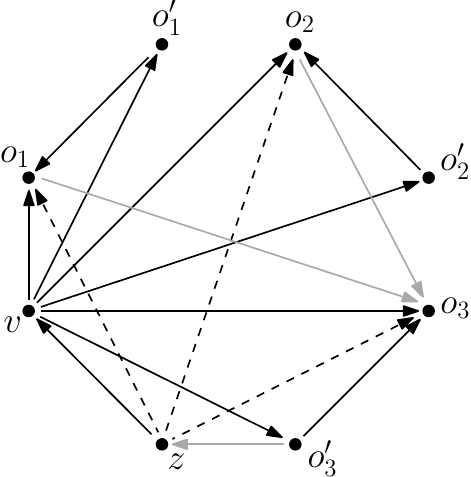}
\end{center}
\caption{Illustration of Step 1-2 from  Lemma~\ref{thm:eight_one} for the tournament $T$ of order $8$, which has a vertex $v$ of $id(v)=1$; \textit{dashed arrows} denote new added arcs in Step~1 and \textit{gray arrows} denote new added arcs in Step~2.}
\label{fig:example8_one}
\end{figure}

Before we prove the final equality we make an observation about the lower bound of $4$ for $DOM^{*}(K_n)$, for any $n \geq 6$.

\begin{obs}\label{obs:more_than_four} Denote by  $T^6_n$ the tournament of order $n+6$ of Figure~\ref{fig:more_than_four}. {\sc Arumugam} et al.~\cite{AES} proved that $\gamma ^*(T^6_0)=\gamma ^*(T^6_1)=4$. So for any $n\geq 0$, $\gamma ^*(T^6_n)=4$ and  $4\leq DOM^{*}(K_{n+6})$ .

\end{obs}

 \begin{figure}[htb] 
\centering
\begin{center}
  \includegraphics[scale=0.6]{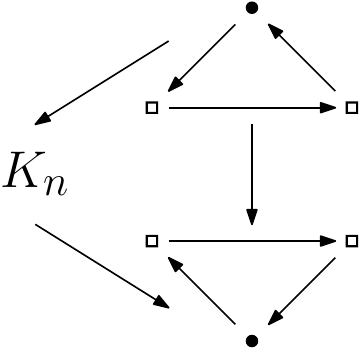}
	\end{center}
\caption{Illustration for Observation~\ref{obs:more_than_four}; \textit{boxes} denote the twin domination set of $T^6_n$ of size $4$, for $n\geq 0$.}
\label{fig:more_than_four}
\end{figure}

We close this section by presenting the main theorem, which summarizes previous results and considers the remaining case, when all vertices of a tournament of order $8$ have in- or outdegree equal to $3$.

\begin{theorem}\label{case8} The upper orientable twin domination number of $K_8$, $DOM^{*}(K_8)= 4$. 
\end{theorem}

\begin{proof} 

Let $T$ be an orientation of $K_8$.  By  Observation~\ref{sink-source} and Lemmas~\ref{thm:eight_two}--\ref{thm:eight_one}, if there exist $v\in V(T)$ such that $id(v)\in \{0,1,2,5,6,7\}$, then $\gamma ^* (T)\leq 4$. So, we can suppose that for every $v\in V(T)$, $id(v)=3$ or $id(v)=4$.

Let $v\in V(T)$ such that $O(v)=\{ o_1,o_{1}',o_2,o_{2}'\}$ and $I(v)=\{ i,i',z\} $. Without loss of generality, we can suppose that the arcs $(i,i'), (o_1',o_1), (o_2',o_2)\in A(T)$. 

If one of the arcs $(i,z), (o_1,z)$ or $(o_2,z)$ are in $A(T)$, then $S=\{ v,i,o_1,o_2\}$ is a twin dominating set of $T$. So, we can suppose that $O(z)=\{ v, i, o_1, o_2\}$ and $I(z)=\{ i', o_1',o_2'\}$. 

\begin{enumerate}
\item If $(o_1,o_2)\in A(T)$, then $S=\{ v,z,i,o_2\}$ is a twin dominating set of $T$.
\item If $(o_2,o_1)\in A(T)$, then $S=\{ v,z,i,o_1\}$ is a twin dominating set of $T$.    
\end{enumerate}

The case when $id(v)=4$ is symmetric. Therefore, $DOM^{*}(K_8)\leq 4$. By Observation~\ref{obs:more_than_four}, we can conclude that $DOM^{*}(K_8)=4.$

\end{proof}

\section{$DOM^*(K_9) = 4$}\label{sec:for9}

In this section we prove the exact value of the upper orientable twin domination number of the graph $K_9$ to be equal to $4$, i.e., $DOM^*(K_9) = 4$. First, let us make a simple observation, that will be very useful in the rest of the section. 

\begin{obs}\label{obs:equal}
For every tournament $T$ of order $n$, since 
\[\sum\limits_{v \in V(T)} id(v) =
\sum\limits_{v \in V(T)} od(v) = \frac{n(n-1)}{2},\]
there exist vertices $u_1,u_2 \in V(T)$, such that $od(u_1), id(u_2) \leq \left\lfloor\frac{n-1}{2} \right\rfloor$.
\end{obs}

Similarly, as for the previous case of the number of vertices $8$, we show that our result holds if a graph contains at least one vertex with a specified in- or out-degree. Observation~\ref{sink-source} assures us that if a tournament of order $9$ contains a sink or source vertex then $DOM^*(K_9) \leq 4$. In the next lemma we prove that this result holds also when a tournament contains a vertex of in- or outdegree $2$.

\begin{lemma} \label{thm:nine_two} Let $T$ be an orientation of $K_9$. If there exists $v\in V(T)$ such that $ id (v) = 2$ or $od (v) = 2$, then $\gamma ^{*}(T)\leq 4$. 
\end{lemma}

\begin{proof}
Suppose there exists a vertex $v \in V(T)$ such that $id_T(v)=2$. Denote the tournament induced by the outset of $v$ as $T_1 = \cG[O_T(v)]$. By Observation~\ref{obs:equal} there exists a vertex $v_1 \in V(T_1)$, such that $od_{T_1}(v_1) \leq 2$ (because $|V(T_1)| = 6$). Let $T_2 = \cG[O_{T_1}(v_1)]$. Let $v'_1 \in I_T(v)$ be a vertex such that $od_{\cG[I_T(v)]}(v'_1) = 1$, and let $v_2 \in V(T_2)$ be such that $id_{T_2}(v_2) = 1$ if $|V(T_2)| = 2$ or simply $\{v_2\} = V(T_2)$ if $|V(T_2)| = 1$. Notice now that $S =\{ v, v_1, v'_1, v_2\}$ is a twin dominating set of $T$. Indeed,
\begin{itemize}
\item $v$ in-dominates $V(T_1)\cup\{v\}$, $v'_1$ in-dominates $I_T(v)$ and 
\item $v$ out-dominates $V(T)\backslash V(T_1)$, $v_1$ out-dominates $V(T_1)\backslash V(T_2)$ and $v_2$ out-dominates $V(T_2)$.
\end{itemize}
See Figure~\ref{fig:example9_two} for the illustration. Notice that if $V(T_2)$ is the empty set, then $S = \{v, v_1, v'_1\}$ is a twin dominating set. Therefore, if $id_T(v) = 2$, $\gamma ^{*}(T)\leq 4$. The case $od_T(v) = 2$ is symmetric.
\end{proof}

\begin{figure}[htb]
\begin{center}
\centering
  \includegraphics[scale=0.9]{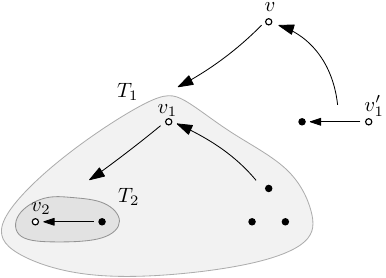}
	\end{center}
\caption{Illustration for Lemma~\ref{thm:nine_two} for the tournament $T$ of order $9$, which has a vertex $v$ of $id(v)=2$; \textit{circles} denote the twin dominating set of $T$.}
\label{fig:example9_two}
\end{figure}

In the next considered case a tournament of order $9$ contains a vertex of in- or outdegree $1$.

\begin{lemma} \label{thm:nine_one} Let $T$ be an orientation of $K_9$. If there exists $v\in V(T)$ such that $ id (v) = 1$ or $od (v) = 1$, then $\gamma ^{*}(T)\leq 4$. 
\end{lemma}

\begin{proof}
Suppose there exists a vertex $v \in V(T)$ such that $id_T(v)=1$. We denote the tournament induced by the outset of $v$ as $T_1 = \cG[O_T(v)]$ and $\{v'_1\} = I_{T}(v)$. By Observation~\ref{obs:equal} there exists a vertex $v_1 \in T_1$ such that $od_{T_1}(v_1) \leq 3$. Let $T_2 = \cG[O_{T_1}(v_1)]$. Repeating this reasoning, we pick a vertex $v_2 \in V(T_2)$ such that $od_{T_2}(v_2) \leq 1$ and let $T_3 = \cG[O_{T_2}(v_2)]$.
While $|V(T_2)| \leq 3$ and $|V(T_3)| \leq 1$ several cases have to be considered. If $|V(T_2)| = 0$, then $S = \{v,v_1,v'_1\}$ is a twin domination set of $T$. If $|V(T_2)| \in  \{1, 2\}$ then $S = \{v,v_1,v'_1, u\}$, where $u \in V(T_2)$ such that $id_{T_2}(u) = 1$ if $|V(T_2)| = 2$ or simply $\{u\} = V(T_2)$ if $|V(T_2)| = 1$, is a twin domination set of $T$. Finally, if $|V(T_2)| = 3$ and $|V(T_3)| = 0$, then $S = \{v,v_1,v'_1,v_2\}$ is a twin domination set of $T$. Therefore, assume $|V(T_2)| = 3$ and $|V(T_3)| = 1$ and let $\{v_3\} = V(T_3)$.

\begin{enumerate}
\item If $(v_1,v_1'), (v_2,v'_1)$ or $(v_3,v_1')$ are in $A(T)$, then $S = \{v,v_1,v_2,v_3\}$ is a twin domination set of $T$. So let us assume that $(v_1',v_1), (v'_1,v_2), (v_1',v_3) \in A(T)$.
\item If $(s, v_1') \in A(T)$, where $\{s\} = I_{T_2}(v_2)$, then $S = \{v,v_1,v'_1,v_3\}$ is a twin domination set of $T$. So let us assume that $(v_1', s) \in A(T)$.
\end{enumerate}

If $id_T(v_1') \in \{0,2\}$, then from Observation~\ref{sink-source} and Lemma~\ref{thm:nine_two} we have $\gamma ^{*}(T)\leq 4$. Therefore, we consider two remaining cases.

\noindent
\emph{Case 1: $id_T(v_1') = 1$.} Let $c_1, c_2$ and $c_3$ denote the vertices, which form an oriented cycle in $\cG[I_{T_1}(v_1)]$. Without loss of generality we assume $(v'_1,c_1), (c_2,v'_1), (v_1', c_3) \in A(T)$. It can be observed now, that $S = \{v'_1,v_2,v_3,c_1\}$ is a twin domination set of $T$. See Figure~\ref{fig:example9_oneA}.

\noindent
\emph{Case 2: $id_T(v_1') = 3$.} We immediately obtain a twin dominating set of $T$ as $S = \{v, v'_1,v_2,v_3\}$. See Figure~\ref{fig:example9_oneB}. 

Therefore, if $id_T(v) = 1$, $\gamma^*(T) \leq 4$. The case $od_T(v) = 1$ is symmetric.
\end{proof}

  \begin{figure}[htb]
\centering
\begin{subfigure}{.47\textwidth}
  \centering
  \includegraphics[width=\textwidth]{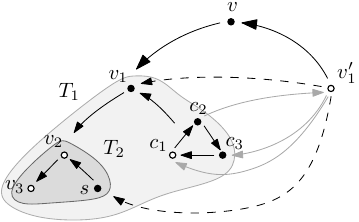}
  \caption{Case~1: $id_T(v'_1) = 1$.}
  \label{fig:example9_oneA}
\end{subfigure}~~
\begin{subfigure}{.47\textwidth}
  \centering
  \includegraphics[width=\textwidth]{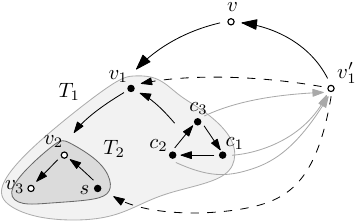}
  \caption{Case~2: $id_T(v'_1) = 3$.}
  \label{fig:example9_oneB}
\end{subfigure}
\caption{Illustration of Case~1-2 from Lemma~\ref{thm:nine_one} for the tournament $T$ of order $9$, which has a vertex $v$ of $id(v)=1$; \emph{dashed arrows} denote new added arcs in Step~1-2 and \emph{gray arrows} denote new added arcs in each case.}
\label{fig:example9_one}
\end{figure}

We continue our reasoning by looking at a tournament of order $9$ that contains a vertex of in- or outdegree $3$.

\begin{lemma} \label{thm:nine_three} Let $T$ be an orientation of $K_9$. If there exists $v\in V(T)$ such that $ id (v) = 3$ or $od (v) = 3$, then $\gamma ^{*}(T)\leq 4$. 
\end{lemma}

\begin{proof}
Suppose there exists a vertex $v \in V(T)$ such that $id_T(v)=3$. We denote the tournaments induced by the out- and insets of $v$ as $T_1 = \cG[O_T(v)]$ and $T'_1 = \cG[I_T(v)]$ respectively. By Observation~\ref{obs:equal} there exist vertices $v_1 \in T_1$ and $v_1' \in T'_1$, such that $od_{T_1}(v_1) \leq 2$ and $id_{T'_1}(v_1') \leq 1$. Denote $T_2 = \cG[O_{T_1}(v_1)]$ and $T'_2 = \cG[I_{T'_1}(v'_1)]$. While $|V(T_2)| \leq 2$ and $|V(T'_2)| \leq 1$ several cases have to be considered. Notice first that if $|V(T_2)| = |V(T'_2)| = 0$, then $S = \{v, v_1, v'_1\}$ is a twin dominating set of $T$. On the other hand, if $|V(T_2)| = 0$ and $|V(T'_2)| \neq 0$ ($|V(T'_2)| = 0$ and $|V(T_2)| \neq 0$), then set $S = \{v, v_1, v'_1, u\}$, where $u \in V(T'_2)$ ($u \in V(T_2)$, respectively) is a twin dominating set of $T$. Therefore, there are two remaining cases.

\noindent
\emph{Case 1: $|V(T_2)| = |V(T'_2)| = 1$.} Let $\{v_2\} = V(T_2)$, $\{v_2'\} = V(T'_2)$ and $\{s\} = O_{T'_1}(v'_1)$. 
\begin{enumerate}
\item If $(v_2,v_2'), (v_2,v'_1)$ or $(v_1,v_2')$ are in $A(T)$, then $S = \{v,v_1,v_1',v_2\}$ or $S = \{v,v_1,v_1',v_2'\}$ are twin domination sets of $T$. So let us assume that $(v_2',v_2), \allowbreak (v'_1,v_2), (v_2',v_1) \in A(T)$. See Figure~\ref{fig:example9_threeA}.
\item If $(s,v_1) \in A(T)$, then $od(v_1) \leq 2$ and from previous lemmas $\gamma^*(T) \leq 4$. On the other hand, if $(v_1,s) \in A(T)$, then $S = \{v,v_1,v_2,v_2'\}$ is a twin domination set of $T$.
\end{enumerate}

\noindent
\emph{Case 2: $|V(T_2)| = 2$ and $|V(T'_2)| = 1$.} Let $\{v_2'\} = V(T'_2)$, $\{s\} = O_{T'_1}(v'_1)$ and $v_2 \in V(T_2)$, such that $id_{T_2}(v_2) = 1$. 
\begin{enumerate}
\item If $(v_2,v_2')$ or $(v_1,v_2')$ are in $A(T)$, then $S = \{v,v_1,v_1',v_2\}$ is a twin domination set of $T$. So let us assume that $(v_2',v_2), (v_2',v_1) \in A(T)$.
\item If $(v_1,s), (v_2,s)$ or $(v'_2,s)$ are in $A(T)$, then $S = \{v,v_1,v_2,v'_2\}$ is a twin domination set of $T$. So let us assume that $(s,v_1), (s,v_2), (s,v'_2) \in A(T)$. See Figure~\ref{fig:example9_threeB}.
\item If $(v'_1,v_1) \in A(T)$, then $od_T(v_1) = 2$ and from previous lemma $\gamma^*(T) \leq 4$. On the other hand, if $(v_1,v'_1) \in A(T)$, then $S = \{v,v_1,v_2,s\}$ is a twin domination set of $T$.
\end{enumerate}
Therefore, if $id_T(v) = 3$, $\gamma^*(T) \leq 4$. The case $od_T(v) = 3$ is symmetric.
\end{proof}

  \begin{figure}[htb]
\centering
\begin{subfigure}{.47\textwidth}
  \centering
  \includegraphics[width=\textwidth]{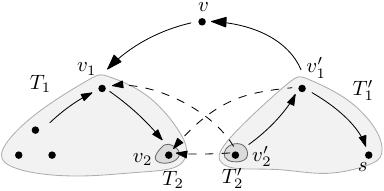}
  \caption{Case 1: Step 1; \emph{dashed arrows} denote new added arcs.}
  \label{fig:example9_threeA}
\end{subfigure}~~
\begin{subfigure}{.47\textwidth}
  \centering
  \includegraphics[width=\textwidth]{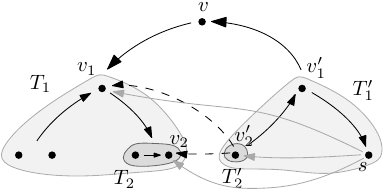}
  \caption{Case 2: Steps 1-2; \emph{dashed arrows} denote new added arcs in Step $1$ and \emph{gray arrows} denote new added arcs in Step $2$.}
  \label{fig:example9_threeB}
\end{subfigure}

\caption{Illustration of Case 1-2 from Lemma~\ref{thm:nine_three} for the tournament $T$ of order $9$, which has a vertex $v$ of $id(v)=3$.}
\label{fig:example9_three}
\end{figure}

We close this section by presenting the main theorem, which summarizes previous results and considers the remaining case, when all vertices of a tournament of order $9$ have indegree equal to $4$.

\begin{theorem}\label{case9} The upper orientable twin domination number of $K_9$, $DOM^{*}(K_9)= 4$. 
\end{theorem}

\begin{proof} 
Let $T$ be an orientation of $K_9$.  By Observation~\ref{sink-source} and Lemmas~\ref{thm:nine_two}--\ref{thm:nine_three}, if there exists $v\in V(T)$ such that $id_T(v)\in \{0,1,2,3,5,6,7,8\}$, then $\gamma ^* (T)\leq 4$. So, we can suppose that for every $v\in V(T)$, $id_T(v)=4$.

Let $v \in V(T)$ be any vertex. We denote the tournaments induced by the out- and insets of $v$ as $T_1 = \cG[O_T(v)]$ and $T'_1 = \cG[I_T(v)]$ respectively. By Observation~\ref{obs:equal} there exist vertices $v_1 \in T_1$ and $v_1' \in T'_1$, such that $od_{T_1}(v_1), id_{T'_1}(v_1') \leq 1$. Repeating the reasoning from the proof of the previous lemma, we observe that if $od_{T_1}(v_1) = 0$ or $id_{T'_1}(v_1') = 0$, then $\gamma ^* (T)\leq 4$. Therefore, let $od_{T_1}(v_1) = 1$ and $id_{T'_1}(v_1') = 1$ and we denote these vertices as $\{v_2\} = O_{T_1}(v_1)$ and $\{v'_2\} = I_{T'_1}(v'_1)$ respectively. If $(v_2,v_2'), (v_2,v'_1)$ or $(v_1,v_2')$ are in $A(T)$, then $S = \{v,v_1,v_1',v_2\}$ or $S = \{v,v_1,v_1',v_2'\}$ are twin domination sets of $T$. So let us assume that $(v_2',v_2), (v'_1,v_2), (v_2',v_1) \in A(T)$. Now, because $v_2$ (resp. $v_2'$) has already four incoming (resp. outcoming) arcs, all of the rest arcs have to be outcoming (resp. incoming), which means that the set $S = \{v_2,v'_2\}$ is a twin dominating set of $T$.
Therefore, $DOM^{*}(K_9)\leq 4$. By  Observation~\ref{obs:more_than_four}, we can conclude that $DOM^{*}(K_9)=4.$
\end{proof}

\section{The Linear Upper Bound of $DOM^*(K_n)$}\label{sec:even}

In this section we prove that for every integer  $n\geq 8$, the conjecture given in {\sc Arumugam} et al.~\cite{AES} is false.

\begin{theorem}\label{thm:max_for_even}
For every integer $n \geq 8$, the upper orientable twin domination number of $K_n$, $DOM^{*}(K_n) \leq \left\lfloor\frac{n}{2}\right\rfloor$. 
\end{theorem}

\begin{proof} 
We prove first the theorem for even $n$.

Let $n=2k$, $k\geq 4$. We use induction on $k$. 

If $k=4$, the theorem holds by Theorem~\ref{case8}. 

Suppose the theorem is true for any tournament $T_1$ of order $2k$. Let $T_2$ be a tournament  of order $2k+2$. Consider $v_1,v_2\in V(T_2),\ (v_1,v_2)\in A(T_2)$ and $T_1$ the subtournament of $T_2$ induced by $V(T_1)=V(T_2)-\{v_1,v_2\}$. 

By our induction hypothesis, there exist a twin domination set $S_1$ of $T_1$, such that $|S_1| \leq k$. If $|S_1| < k$, then $S_2= S_1\cup \{ v_1,v_2\}$ is a twin dominating set of $T_2$ with $|S_2|\leq k+1$. So, suppose  $|S_1| = k.$

\begin{enumerate}
\item If for one vertex $v \in S_1$, $(v, v_1) \in A(T_2)$ or $(v_2,v) \in A(T_2)$, then the set $S_2 = S_1 \cup \{ v_2\} $ or $S_2 = S_1 \cup \{ v_1\} $, respectively, is a twin dominating set of $T_2$ with cardinality $|S_2|=k+1$. Therefore we can assume that $v_1$ is a source and $v_2$ a sink with respect to the set $S_1$. Notice that $\{ v_1, v_2\} $ is a twin dominating set of $S_1$.

\item If one of the vertices in $T_2$ is source or sink the proof is finished according to Observation~\ref{sink-source}.  
So there  exists $o_1,o_2\in V(T_2)$ such that the arcs $(o_1,v_1), (v_2,o_2) \in A(T_2)$. If $o_1=o_2$, $v_1v_2o_2$ or $o_1v_1v_2$ is an oriented cycle in $T_2$, then $S_2= V(T_2) - \{ S_1\cup \{o_i\}\}$ with $i\in \{1,2\}$ is a twin dominating set of $T_2$ with $|S_2|=k+1$.  Thus, the arcs  $(o_1,v_2), (v_1,o_2) \in A(T_2)$. It is clear, that the set  $S_2 = V(T_2) - S_1$ is a twin dominating set of $T_2$ with cardinality $|S_2|=k+2$. We will prove that $S_2$ it is not minimum. Let $v \in S_2 - \{v_1, v_2, o_1, o_2 \}$.

\begin{enumerate}

\item If the arc $(o_2,o_1) \in A(T_2)$, then $S_2-\{ o_2\}$ is a twin dominating set of $T_2$ with cardinality $k+1$. So, we can suppose we have the arc $(o_1,o_2) \in A(T_2)$. See Figure~\ref{fig:even_proof}.

\item If the arc $(v,o_1)$ or $(o_2,v)$ is in $A(T_2)$, then $S_2-\{ o_1\}$ or $S_2-\{ o_2\}$ are twin dominating set of $T_2$, respectively. So we can assume that in $T_2$, we have the arcs $(o_1,v),(v,o_2)$, which implies that $S_2-\{ v\}$ is a twin dominating set of $T_2$ with cardinality $k+1$. 
\end{enumerate}
\end{enumerate}

The reasoning for odd numbers, i.e., $n = 2k + 1$ for $k \geq 4$, is similar, as the theorem holds for $k = 4$ by Theorem~\ref{case9}.
\end{proof}

 \begin{figure}[htb]
\centering
\begin{center}
  \includegraphics[scale=0.6]{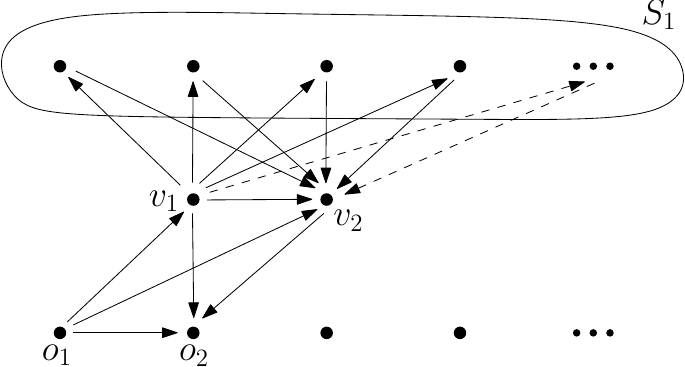}
	\end{center}
\caption{Tournament $T_2$; \textit{black dots} denote vertices from $V(T_2)$, where $V(T_1) = V(T_2)- \{v_1,v_2\}$ and $S_1$ denotes the twin domination set of $T_1$.}
\label{fig:even_proof}
\end{figure}

\section{The Logarithmic Upper Bound of $DOM^{*}(K_n)$}\label{sec:main}

In the following section we prove the $O(\log_2(n))$ upper bound of $DOM^{*}(K_n)$ for any $n \geq 4$.

\begin{theorem}\label{thm:upper_bound}
For every natural number $n \geq 4$, the upper orientable twin domination number of $K_n$, $DOM^{*}(K_n) \leq 2 \left\lceil\log_2{(n-1)} \right\rceil$. 
\end{theorem}

\begin{proof}
Let $T$ be any orientation of $K_n$ and $u \in V(T)$ any vertex.  
Let $T_1$ be the subtournament of $T$ induced by $O_T(u)$. As $u$ in-dominates $T_1$ there has to be found an out-dominating set in $T_1$, which is done in the following recursive way. Let $u_i \in T_i,\ i \in \natplus$ such that $od_{T_i}(u_i) \leq \left\lfloor\frac{n_i-1}{2}\right\rfloor$ (the existence of such a vertex is guaranteed by Observation~\ref{obs:equal}), where $n_i = |V(T_i)|$, then a recursive equations are 
\begin{align}\nonumber
T_1 &= \cG[O_T(u)], \\ \nonumber
T_i &= \cG[O_{T_{i-1}}(u_{i-1})],\ i=2,\ldots, l;
\end{align}
where $l$ is a first integer for which $|V(T_l)| \leq 2$. Set $S_1 = \{u_i | i=1,\ldots,l\}$ out-dominates $T_1$ and while the number of vertices at each step decreases twice $|S_1| \leq \left\lceil\log_2{od(u)}\right\rceil$. Similarly, we construct the in-dominating set $S_2$ for a subtournament of $T$ induced by $I_T(u)$, of the maximum cardinality $|S_2| \leq \left\lceil\log_2{id(u)} \right\rceil$. Set $S = S_1 \cup S_2 \cup \{u\}$ twin dominates $T$ and 
\begin{align}\nonumber
|S| &= |S_1| + |S_2| + 1 \leq \left\lceil\log_2{od(u)} \right\rceil + \left\lceil\log_2{id(u)} \right\rceil + 1 \leq \\ \nonumber
& \leq \left\lceil\log_2{\left(od(u)id(u)\right)} \right\rceil	+ 2 \leq 2 \left\lceil\log_2{\frac{n-1}{2}} \right\rceil + 2 = 2 \left\lceil\log_2{(n-1)} \right\rceil.
\end{align}
\end{proof}

This logarithmic upper bound is tighter than the linear one for $n > 21$, i.e., $2 \left\lceil\log_2{(n-1)} \right\rceil < \left\lfloor\frac{n}{2}\right\rfloor$ for $n > 21$. We summarize all results for $DOM^*(K_n)$ from {\sc Arumugam} et al.~\cite{AES} and this paper  in the following theorem.

\begin{theorem}\label{thm:main}
For $n \geq 1$,
$$
DOM^{*}(K_n) = \left\{ \begin{array}{ll}
\left\lceil\frac{n+1}{2}\right\rceil & \textrm{if $1 \leq n \leq 7$,}\\
4 & \textrm{if $n \in \{8,9\}$;}
\end{array} \right.
$$
$$
DOM^{*}(K_n) \leq \left\{ \begin{array}{ll}
\left\lfloor\frac{n}{2}\right\rfloor & \textrm{if $9 < n \leq 21$,} \\
2 \left\lceil\log_2{(n-1)} \right\rceil& \textrm{if $21 < n$.}
\end{array} \right.
$$
\end{theorem}

\medskip
\noindent
Research partially supported by Grant UNAM-PAPIIT-IN-114415 and National Science Centre (Poland) grant number 2015/17/B/ST6/01887.

\medskip
\noindent
The authors would like to thank G. Karolyi for his valuable suggestion on Theorem~\ref{thm:upper_bound}.

\end{document}